\documentclass[a4paper,reqno,11pt]{amsart}
\usepackage[a4paper,margin=1in]{geometry}
\usepackage{amsmath,amssymb,amsthm}
\theoremstyle{plain}
\newtheorem{theorem}{Theorem}[section]
\newtheorem{lemma}[theorem]{Lemma}
\newtheorem{proposition}[theorem]{Proposition}

\newtheorem{corollary}[theorem]{Corollary}
\theoremstyle{definition}
\newtheorem{remark}[theorem]{Remark}
\numberwithin{equation}{section}
\title[Loops and monotonicity]{On loops in water wave branches and monotonicity of water waves}

\author{Vladimir Kozlov$^1$}
\begin{document}
	
\begin{abstract} We give a quite simple approach how to prove the absence of loops in bifurcation branches of water waves in rotational case with
 arbitrary vorticity distribution. The supporting flow may contain stagnation
points and critical layers and water surface is
allowed to be overhanging. Monotonicity properties of the free surface are presented.
Especially simple criterium of absence of loops is given for bifurcation branches when the bifurcation parameter is the water wave period. We show that there are no loops if you start from a water wave with  a positive/negative vertical component of velocity on the positive half period.

\end{abstract}

\maketitle

\section{Introduction}

We consider steady water waves on vortical flows. We admit existence of stagnation points and critical layers inside the flow and the water surface is allowed to be overhanging.

A standard way to prove the existence of large amplitude water waves is a construction of bifurcating branches of solutions which start from a trivial-horizontal wave and then approach a certain limit configuration. %try to understand how long we can continue such branches.
We will use here analytic bifurcation theory developed in \cite{BTT}.  According to the analytic theory there are several limit options:

(a) a stagnation point on the free surface;

(b) self-intersection of the free surface, when the overhanging is allowed;

(c) the free surface can touch the bottom.

(d) constants of the problem, which depend on the bifurcation parameter can go to infinity or zero.

There is another option: existence of a closed curve.
Usually this option can be analysed and excluded by using a nodal analysis, see \cite{CSst} for the case of unidirectional flows and \cite{CSrVar} for the case of constant vorticity. The aim of this paper is to study this option in different cases, when the water waves and supporting flow may  contain stagnation points and critical layers and when the free surface is allowed to be overhanging, see \cite{Var23} and \cite{Wa} for discussions of problems arising in the nodal analysis.

We prove that there are no loops if the closed branch may contain only one uniform stream solution. In the case when the bifurcation parameter is the water wave period, we show that there are no closed branches at all. This is an interesting difference betweed branches with fixed period and when the bifurcation parameter is chosen as the period. The first advantage of the choice of the period as a bifurcation parameter is the fact that the transversality condition, required for existence of small amplitude water waves, is satisfied automatically. The second one there are no loops on the analytic branches of water waves.
This means that analytic branches of water waves are unbounded and we can reach realy large water waves along such branches.

Let us explain the main idea of our approach. The problem is formulated in the natural Cartesian  coordinates $(x,y)$ and the nodal analysis is based on the analysis of the sign of the vertical component of the velocity. The geometry of water domain is not fixed in these variables, it is unknown. To construct branches of water waves another variables are used in which the water domain is fixed. This allows to apply general theorems on existence of global branches of water waves which analytically depend on a bifurcation parameter. We use the variables $(x,y)$ for application of maximum principle to the vertical component of the velocity and the second set of variables for analysis of continuous and analytical properties of the vertical component  along the branch of solutions.

\subsection{Formulation of the problem}

We consider steady surface waves in a two-dimensional channel bounded below by a flat,
rigid bottom and above by a free surface that does not touch the bottom. The surface tension is neglected and the water motion can be rotational.
In appropriate Cartesian coordinates $(x, y )$, the bottom ${\mathcal B}$ coincides with the line
$y=0$ and gravity acts in the negative $y$ -direction. We choose the frame of reference so that the velocity field is time-independent as well as the free-surface ${\mathcal S}$, which is located in the half-plane $y>0$ and given in a parametric form $x=u(s)$, $y=v(s)$, $s\in \Bbb R$. It is assumed that $|v'|+|u'|\neq 0$,
$v(s)>0$ and $u(s)\to\pm\infty$ when $s\to\pm\infty$. We assume also that the curve ${\mathcal S}$ does not intersect itself. We denote  the strip-like domain between ${\mathcal B}$ and ${\mathcal S}$ by ${\mathcal D}$. The domain ${\mathcal D}$ is assumed to be symmetric with respect to the vertical line $x=0$ and $\Lambda$-periodic with respect to $x$, where $\Lambda$ is a positive number. If we introduce the set
$$
\Omega=\{(x,y)\in{\mathcal D}\,:\,-\Lambda/2<x\leq\Lambda/2\},
$$
then
$$
{\mathcal D}=\bigcup_{j=-\infty}^\infty \Omega_j,\;\;\Omega_j=\{(x+j\Lambda,y)\,:\,(x,y)\in\Omega \}.
$$
The dependence of $u$ and $v$ is also periodic in $s$ but the period can be different from $\Lambda$. If,  for example,  $s$ is the arc length along ${\mathcal S}$, measured from the point $(0,v(0))$ (it is supposed here that $u(0)=0$ and the length comes with the sign $+$ for $s>0$ and with sign $-$ for $s<0$),  the period of the functions $u$ and $v$ is equal to the length of the free surface from $(0,v(0))$ to $(u(s),v(s))$.
We assume that ${\mathcal D}$ is simple connected, i.e. there are no wholes inside ${\mathcal D}$.

 To describe the flow inside ${\mathcal D}$, we will use the stream function $\psi$, which is connected with the velocity vector $({\bf u},{\bf v})$ by ${\bf u}=-\psi_y$ and ${\bf v}=\psi_x$. Since the surface tension is neglected, the function $\psi$  after a certain scaling, satisfies the following free-boundary problem (see for example \cite{KN14}):
\begin{eqnarray}\label{K2a}
&&\Delta \psi+\omega(\psi)=0\;\;\mbox{in ${\mathcal D}$},\nonumber\\
&&\frac{1}{2}|\nabla\psi|^2+y-h=Q\;\;\mbox{on ${\mathcal S}$},\nonumber\\
&&\psi=0\;\;\mbox{on ${\mathcal S}$},\nonumber\\
&&\psi=-m\;\;\mbox{for $y=0$},
\end{eqnarray}
where $\omega$ is a vorticity function, $Q$ is the Bernoulli constant and ${\mathcal D}$ is an unknown domain.
We always assume that the function $\psi$ is even and $\Lambda$ periodic with respect to $x$ and
\begin{equation}\label{J27ba}
\nabla\psi \neq 0\;\;\mbox{on ${\mathcal S}$,}
\end{equation}
which means that there are no stagnation point  on the surface ${\mathcal S}$.

\section{Main Lemma}\label{SAu14c}

Let
$$
D=\{(x,y)\in {\mathcal D}\,:\, 0<x<\Lambda/2,\}\;\;S=\{(x,y)\in {\mathcal S}\,:\, 0<x<\Lambda/2\}.
$$
Then the part of the boundary $\partial {\mathcal D}\cap \{x=0\}$ is an interval $[0,y_0]$ and the part of the boundary
$\partial {\mathcal D}\cap \{x=\Lambda/2\}$ is an interval $[0,y_1]$.
The solution $\psi$ solves the problem
\begin{eqnarray}\label{KK2a}
&&\Delta \psi+\omega(\psi)=0\;\;\mbox{in $D$},\nonumber\\
&&\frac{1}{2}|\nabla\psi|^2+y-h=Q\;\;\mbox{on $S$},\nonumber\\
&&\psi=0\;\;\mbox{on $S$},\nonumber\\
&&\psi=-m\;\;\mbox{for $y=0$}
\end{eqnarray}
 and
\begin{equation}\label{Au3a}
\psi_x(0,y)=0\;\;\mbox{on $(0,y_0)$ and}\;\;\psi_x(\Lambda/2,y)=0\;\;\mbox{on $(0,y_1)$}.
\end{equation}

We always assume that at least $\omega\in C^{1,\alpha}$ with certain $\alpha\in (0,1)$. This implies, in particular that $\psi\in C^{3,\alpha}(\overline{\mathcal D})$.

\begin{lemma}\label{L1} Let $\psi\in C^{3}(\overline{\mathcal D})$ and let  $\psi$ solves the problem {\rm (\ref{KK2a})}, {\rm (\ref{Au3a})}. If $\psi_x\geq 0$ in $D$ and $\psi_x$ is not identically zero, then the following assertions are valid for the function $u=\psi_x$:

\begin{eqnarray}\label{Au16a}
&&(i)\;\;\;\;u>0\;\; \mbox{for $(x,y)\in D\cup S$}\nonumber\\
&&(ii)\;\;\;u_x(0,y)>0\;\; \mbox{for $0<y\leq y_0$,\;\; $u_x(\Lambda/2,y)<0$ for $0<y\leq y_1$}\nonumber\\
&&\;\;\;\;\;\;\;\;\mbox{ and \;\;$u_y(x,0)>0$\; for $0<x<\Lambda/2$}\nonumber\\
&&(iii)\;\;u_{xy}(0,0)>0,\;\; u_{xy}(\Lambda/2,0)<0.
\end{eqnarray}

\end{lemma}
\begin{proof} Since
\begin{equation}\label{Au3ba}
\Delta\psi_x+\omega'(\psi)\psi_x=0\;\;\mbox{in $D$},
\end{equation}
the assertion (i) for the domain $D$ and the inequalities
\begin{eqnarray*}
&&\psi_{xx}(0,y)>0\;\;\mbox{ for $0<y< y_0$},\;\; \psi_{xx}(\Lambda/2,y)<0\;\;\mbox{ for $0<y< y_1$} \\ &&\mbox{and}\;\;\psi_{xy}(x,0)>0\;\;\mbox{for}\;\;0<x<\Lambda/2
\end{eqnarray*}
follows from the maximum principle and Hopf lemma (see  \cite{CSst} and \cite{AN}). Let us prove that
\begin{equation}\label{Au3bb}
\psi_{xx}(0,y_0)>0\;\;\;\mbox{and}\;\; \psi_{xx}(\Lambda/2,y_1)<0.
\end{equation}
Consider the first inequality. The tangent to the boundary curve at the point $(0,y_0)$ is parallel to the $x$-axis and so the curve  near this point can be described by the equation $y=f(x)$ with smooth $f$ and with $f'(0)=0$. The Bernoulli equation on this piece of the curve has the form
$$
\frac{1}{2}|\nabla\psi|^2+f(x)-h=Q\;\;\mbox{for $y=f(x)$}.
$$
Differentiating this relation with respect to $x$ and using that $\psi(x,f(x))=0$, which implies $f'=-\psi_x/\psi/y$, we get
\begin{equation}\label{Au3b}
\partial_\nu(\psi_x)-\rho \psi_x=0\;\;\mbox{for $y=f(x)$}.
\end{equation}
Here $\nu$ is the outward unit normal and
  \begin{equation}\label{Au4a}
\rho=
\rho(x)=\frac{(1+\psi_x\psi_{xy}+\psi_y\psi_{yy})}{\psi_y(\psi_x^2+\psi_y^2)^{1/2}}\Big|_{y=f(x)}.
\end{equation}
By (\ref{Au3a}), (\ref{Au3ba}) and (\ref{Au3b}) we conclude that the asymptotics of $\psi_x$ near the point $(0,y_0)$ is $\psi(x,y)=cx+O(r^2)$, where $c\geq 0$ and $r^2=x^2+(y-y_0)^2$. If $c=0$ then
$$
\psi_x(x,y)=Cp_{n}(x,y-y_0)+O(r^{n+1}),
$$
where $p$ is a harmonic polynomial of two variables  $x$ and $y-y_0$ of order  $n$, $n>2$, which is even with respect to $y-y_0$ and vanishes for $x=0$. All such harmonic polynomials change sign inside $D$ and since $\psi_x\geq 0$ inside $D$, we have that $C=0$. This implies that $\psi$ has zero of infinite order at $(0,y_0)$. By unique continuation property we conclude that $\psi$ is identically zero. This contradiction proves that $\psi_x(x,y)=cx+O(r^2)$ with $c>0$. Similarly, we can prove the second inequality in (\ref{Au3bb}). Thus (ii) is proved.

Let us prove (i) for points on $S$. Denote by $K$ the set of points $(x,y)\in \overline{S}$ for which $\psi_y(x,y)=0$. Clearly $K$ is compact and
$$
\psi_x(x,y)^2\geq c_*;=\min_{(x,y)\in \overline{S}}|\nabla\psi|^2.
$$
The right-hand side here is positive by (\ref{J27ba}).
If $\psi_y(x_*,y_*)\neq 0$ then there is an interval $(x_*-\varepsilon,x_*+\varepsilon)$, $\varepsilon>0$, where
the curve can be parameterized by $y=f(x)$ near the point $(x_*,y_*)$. Differentiating equation $\psi(x,f(x))=0$ with respect to $x$ we get the relation (\ref{Au3b}).
Then if $u(x,y)=0$ on this curve at a certain point $(x,y)$, then $u_\nu(x,y)=0$ by (\ref{Au3b}). On the other hand by the Hopf lemma $u_\nu$ must be different from zero. This contradiction proves (i) for $S$. So the assertions (i) and (ii) are proved.

Let us prove (iii). From (\ref{Au3ba}) and (\ref{Au3a}) it follows that
\begin{equation}\label{Au14a}
\psi_x(x,y)=C_0xy+o(r_0^2)\;\;\psi_x(x,y)=C_1(x-\Lambda/2)y+o(r_1^2),
\end{equation}
where $r_0$ and $r_1$ are the distances to $(0,0)$ and $(\Lambda/2,0)$, and $C_0\geq 0$, $C_1 \leq 0$ respectively. Let $C_0=0$. The the asymptotics has the form
$$
\psi_x(x,y)=Cp_n(x,y)+o(r_0^n),
$$
where $p_n$ is a harmonic polynomial of degree $n>2$ vanishing for $x=0$ and for $y=0$. Such polynomials change sign in $D$ and therefore $C=0$, which leads to $\psi_x(x,y)=O(r_0^n)$ for any $n$. This implies $\psi_x=0$. This contradiction proves that $C_0>0$. Similarly one can prove that $C_1<0$.
\end{proof}

\begin{remark} We note that the assertion {\rm (iii)} in the above lemma is equivalent to the  asymptotics (\ref{Au14a}) for $\psi_x$ near the points $(0,0)$ and $(\Lambda/2,0)$,
where  $C_0>0$, $C_1<0$ respectively.
\end{remark}

\begin{corollary} Let $\psi$ be the same as in the previous lemma. Let also on a piece of the curve $S$, $\psi_y\neq 0$. Then this piece can be parameterized as $y=f(x)$. Moreover  $f'(x)\neq 0$ on this part of the curve and it has the same sign as $-\psi_y$.

\end{corollary}
\begin{proof} Since the tangent to the curve is $(-\psi_y,\psi_x)$ it can be parameterized as $y=f(x)$ on any interval where $\psi_y\neq 0$.
Differentiating the relation $\psi(x,f(x))=\mbox{const}$, we get $f'(x)=-\psi_x/\psi_y$. Therefore $f'\neq 0$ and it has the same sign as $-\psi_y$.

\end{proof}

The boundary of $D$ consists of four open segments
$$
L_1=\{0\}\times (0,\eta(0)),\;\;L_2=(0,\Lambda/2)\times\{0\},\;\;L_3=\{\Lambda/2\}\times (0,\eta(\Lambda/2)),\;\;,\;\;L_4=(x,\eta(x)),x\in (0,\Lambda/2),
$$
and four points
$$
A_1=(0,0),\;\;A_2=(\Lambda/2,0),\;\;A_3=(\Lambda/2,\eta(\Lambda/2)),\;\;A_4=(0,\eta(0)).
$$
So
$$
\partial D=\bigcup_{j=1}^4L_j\cup\bigcup_{k=1}^4A_k.
$$

Introduce the domain
$$
R=\{(X,Y)\,:\, 0<X<\Lambda/2,\;a<Y<b\}.
$$
Similarly, we split the boundary of $R$. Let
$$
M_1=\{0\}\times (a,b),\;\;M_2=(0,\Lambda/2)\times\{a\},\;\;M_3=\{\Lambda/2\}\times (a,b)\;\;M_4=(0,\Lambda/2)\times\{b\}
$$
and four points
$$
B_1=(0,a),\;\;B_2=(\Lambda/2,a),\;\;A_3=(\Lambda/2,b),\;\;B_4=(0,b).
$$
Then
$$
\partial R=\bigcup_{j=1}^4M_j\cup\bigcup_{k=1}^4B_k.
$$

Consider a map
$$
F=(F_1,F_2)\,:\,\overline{R}\rightarrow \overline{D}.
$$
We assume that this map is  of class $C^2$  and
\begin{equation}\label{Au6ba}
F(M_j)=L_j\;\;\mbox{and}\;\;F(B_j)=A_j,\;\;j=1,2,3,4.
\end{equation}
Moreover we assume that this map is isomorphism, i.e.
\begin{equation}\label{Au6a}
\frac{\partial x}{\partial X}\frac{\partial y}{\partial Y}-\frac{\partial x}{\partial Y}\frac{\partial y}{\partial X}\neq 0.
\end{equation}

\begin{lemma}\label{L2} Let $u\in C^2(\overline{D})$ satisfy
\begin{equation}\label{Au6aa}
u=0\;\;\mbox{on $\bigcup_{j=1}^3L_j$},
\end{equation}
and properties {\rm (i)-(iii)} from {\rm Lemma \ref{L1}}.
Then the function $v(X,Y)=u(F(X,Y))$ also belongs to $C^2(\overline{\mathcal R})$  and satisfies
\begin{equation}\label{Au6b}
v=0\;\;\mbox{on $\bigcup_{j=1}^3M_j$}
\end{equation}
together with properties {\rm (i)-(iii)} from the same {\rm Lemma}.
\end{lemma}
\begin{proof}  The relation (\ref{Au6b}) follows from (\ref{Au6ba}) and the property (i) can be verified directly.

Let us turn to (ii). Differentiating $v$, we get
$$
\partial_Xv=u_x\frac{\partial x}{\partial X}+u_y\frac{\partial y}{\partial X},\;\;\partial_Yv=u_x\frac{\partial x}{\partial Y}+u_y\frac{\partial y}{\partial Y}.
$$
The function  $u_x$ is equal to zero on $L_1$. Therefore $v_Y=u_y \frac{\partial y}{\partial X}$. Since the Jacobian of $F$ does not vanish, $\frac{\partial x}{\partial Y}\frac{\partial y}{\partial X}\neq 0$ and hence $\frac{\partial y}{\partial X}\neq 0$. Noting that $y(0,a)=0$ and $y(0,b)=\eta(0)$ we conclude that $\frac{\partial y}{\partial X}>0$. Hence $v_Y>0$ on $M_1\cup \{(0,b)\}$.
Similarly one can prove the remaining part of (ii).

To prove (iii) we start from noting that $\nabla v=0$ at the points $(0,a)$ and $(\Lambda/2,a)$. Using this observation, we get
\begin{eqnarray*}
&&v_{XY}=\partial_X(u_x\frac{\partial x}{\partial Y}+u_y\frac{\partial y}{\partial Y})=
(u_{xx}\frac{\partial x}{\partial X}+u_{xy}\frac{\partial y}{\partial X})\frac{\partial x}{\partial Y}+
(u_{yx}\frac{\partial x}{\partial X}+u_{yy}\frac{\partial y}{\partial X})\frac{\partial y}{\partial Y}\\
&&=
u_{xy}\frac{\partial y}{\partial X}\frac{\partial x}{\partial Y}+u_{yx}\frac{\partial x}{\partial X}\frac{\partial y}{\partial Y}.
\end{eqnarray*}

Due to (\ref{Au6ba}), $F_2(M_2)=0$ and $F_1(M_1)=0$. Therefore, $\partial_Xy=\partial_Yx=0$ at $A_1$. Hence
$v_{XY}=u_{xy}\frac{\partial Y}{\partial y}\frac{\partial X}{\partial x}$ at $A_1$. Using (\ref{Au6a}) one can show that $v_{XY}>0$ at $A_1$.

Similar argument proves (iii) at $A_2$.

\end{proof}

\begin{remark} We can assume that $v\in C^2(\overline{R})$ satisfies {\rm (\ref{Au6b})} and {\rm (\ref{Au16a}) (i)-(iii)}. Then the same proof
gives that $u(x,y)=v(X(x,y),Y(x,y))$ satisfies  {\rm (\ref{Au6aa})} together with {\rm (\ref{Au16a}) (i)-(iii)}.

\end{remark}

In the next proposition a stability property of the properties (i)-(iii), (\ref{Au16a}) is verified.

\begin{proposition}\label{PAu17a} Let $u\in C^2(\overline{D})$ satisfy
\begin{equation}\label{Au16b}
u=0\;\;\mbox{on $\bigcup_{j=1}^3L_j$}
\end{equation}
 and  {\rm (i)-(iii)} in {\rm (\ref{Au16a})}, be valid.
Then there exist $\varepsilon>0$ such that if $||u-v||_{C^2(\overline{D})}\leq\varepsilon$ and {\rm (\ref{Au16b})} holds then $v$ also satisfies {\rm (i)-(iii), (\ref{Au16a})}.
\end{proposition}
\begin{proof} Let the properties (i)-(iii), (\ref{Au16a}), are true for s certain $u\in C^2(\overline{D})$ together with (\ref{Au16b}). Then for small $\varepsilon$ and for the function $v$ satisfying (\ref{Au16b}) and $||u-v||_{C^2(\overline{D})}\leq\varepsilon$ all properties in (ii) and (iii) in (\ref{Au16a}) are satisfied. But then the function $v$ is positive in a neighborhood of $L_1\cup L_2\cup L_3$ inside $D$. Then due to (i) the function $v$ is also positive near $L_4$ provided $\varepsilon$ is small. This together with (i) for $D$ implies positivety of $v$ in $D$ for small $\varepsilon$.

\end{proof}

\subsection{Uniform stream solution, dispersion equation}\label{SAu13a}

The uniform stream solution $\psi=\Psi(y)$ with the bottom $y=0$ and the free surface $y=h$,  satisfies the problem
\begin{eqnarray}\label{XX1}
&&\Psi^{''}+\omega(\Psi)=0\;\;\mbox{on $(0,h)$},\nonumber\\
&&\Psi(h)=0,\;\;\Psi(0)=-m
\end{eqnarray}
and the the Bernoulli relation
\begin{equation}\label{J11a}
\frac{1}{2}\Psi'(h)^2=Q.
\end{equation}
To solve the problem (\ref{XX1}), (\ref{J11a}) we start from the following Cauchy problem
\begin{eqnarray}\label{XX1}
&&\Psi^{''}+\omega(\Psi)=0\;\;\mbox{on $(0,h)$},\nonumber\\
&&\Psi(h)=0,\;\;\Psi'(h)=\lambda.
\end{eqnarray}
This problem is uniquely solvable and the solution is  non trivial if $\lambda\neq 0$. Using this solution we can construct solution to (\ref{XX1}), (\ref{J11a}) with $m=-\Psi(0)$ and
$Q=\lambda^2/2$.

We assume that
\begin{equation}\label{J11aa}
\kappa:=\Psi'(h)\neq 0.
\end{equation}
Before introducing the dispersion equation we consider the eigenvalue problem
\begin{equation}\label{J11b}
-w^{''}-\omega'(\Psi)w=\mu w\;\;\mbox{on $(0,h)$},\;\;w(0)=w(h)=0.
\end{equation}
Denote by
$$
\mu_1<\mu_2<\cdots,\;\;\mu_j\to\infty,\;\;\mbox{as $j\to\infty$}
$$
the eigenvalues of the problem (\ref{J11b}), which are simple since we are dealing with one-dimensional problem. Introduce the function $\gamma(y;\tau)$  as solution of the equation
\begin{equation}\label{J16b}
-\gamma^{''}-\omega'(\Psi(y))\gamma+\tau^2\gamma=0,\;\;\mbox{on $(0,h)$}, \,\;\gamma(0;\tau)=0,\;\;\gamma(h;\tau)=1.
\end{equation}
This function is defined for all $\tau$ if $\mu_1>0$ and for $\tau\neq \sqrt{-\mu_j}$ for all $\mu_j\leq 0$.
The frequency of small amplitude Stokes waves is determined by the dispersion equation (see, for example, \cite{KN14} and \cite{Koz1}).
We put
\begin{equation}\label{Au12a}
\sigma(\tau)=\kappa(\gamma'(h;\tau)-\rho_0),
\end{equation}
where
\begin{equation}\label{F28a}
\rho_0=\frac{1+\Psi'(h)\Psi^{''}(h)}{\Psi'(h)^2}.
\end{equation}
The dispersion  equation   is the following
\begin{equation}\label{Okt6bb}
\sigma(\tau)=0.
\end{equation}

We note that
$$
\frac{1+\Psi'(h)\Psi^{''}(h)}{\Psi'(h)^2}=\kappa^{-2}-\frac{\omega(1)}{\kappa}
$$
and hence another form for (\ref{Au12a}) is
\begin{equation}\label{M21aa}
\sigma(\tau)=\kappa\gamma'(d,\tau)-\kappa^{-1}+\omega(1).
\end{equation}
Certainly the function $\sigma$ is defined for the same values of $\tau$ as the function $\gamma$.
The main properties of the function $\gamma$ are presented in the following

\begin{proposition}\label{Pr1} {\rm (i)}
\begin{equation}\label{J16a}
\partial_\tau\gamma'(h;\tau)>0\;\;\mbox{for all $\tau>0$ for which $\gamma$ is defined}.
\end{equation}

%(ii) If the problem (\ref{J11b}) has negative or zero eigenvalue (let $\lambda_n$, $n\geq 1$, be the least such eigenvalue) then the function %$\gamma(d;\tau)$ is increasing on each interval.

{\rm (ii)} If $\mu_1<0$ and $\tau^2 <-\mu_1$ then the function $\gamma(y;\tau)$ changes sign on the interval $y\in (0,h)$.

\end{proposition}
\begin{proof} (i) This assertion is proved in \cite{KN14}, \ Lemma 1.1.

(ii) Let $\tau_1^2=-\mu_1$ and let $\phi_1$ be  a positive eigenfunction on $(0,h)$  corresponding to $\mu_1$. Multiplying the equation in  (\ref{J16b}) by $\phi_1$ and integrating over the interval $(0,h)$, we get
$$
(\tau_1^2-\tau^2)\int_{0}^h\gamma\phi_1dy=\phi_1'(h).
$$
Since $\phi_1'(h)<0$ and $\tau<\tau_1$ we get
$$
\int_{0}^h\gamma\phi_1dy<0.
$$
Since $\gamma(h;\tau)=1$ the function $\gamma$ must change sign on $(0,h)$.
\end{proof}
\begin{corollary}\label{CAu12a} If $\kappa>0$ then the function $\sigma'(\tau)>0$ for $\tau>0$. If $\kappa<0$ then $\sigma'(\tau)<0$ for $\tau>0$.

\end{corollary}
\begin{proof} The proof follows directly from (\ref{J16a}) and the definition of $\sigma$.

\end{proof}

The following properties of the function $\sigma$ are proved in \cite{KN14}

\begin{proposition}\label{Pr2} {\rm (i)}
$$
\sigma(\tau)=\kappa\tau +O(1)\;\;\mbox{as $\tau\to\infty$}.
$$

{\rm (ii)} Let $\tau_*^2$ be an eigenvalue of the problem (\ref{J11b}) and let $\phi_*(y)$ be the corresponding eigenfunction normalized by $||\phi_*||_{L^2(-h,0)}=1$. Then
$$
\sigma(\tau)=-\frac{\kappa\phi'_*(h)^2}{\tau^2-\tau_*^2}+O(1)\;\;\mbox{for $\tau$ close to $\tau_*$}.
$$
We note that due to the homogeneous Dirichlet boundary condition $\phi'_*(h)\neq 0$.
\end{proposition}

From Propositions \ref{Pr1}, \ref{Pr2} and Corollary \ref{CAu12a} we get
\begin{theorem}\label{TAu8} The following assertions are valid

{\rm (i)} Let $\mu_1>0$ then the dispersion equation {\rm (\ref{Okt6bb})} has a positive root if and only if $\sigma(0)<0$ in the case $\kappa>0$ and $\sigma(0)>0$ in the case $\kappa<0$. In both cases the function $\gamma(y;\tau)>0$ for $y\in (0,h]$.

{\rm (ii)} If $\mu_1\leq 0$ then the dispersion equation {\rm (\ref{Okt6bb})} is uniquely solvable. If $\tau_*$ is a root of this equation then the function $\gamma(y;\tau_*)$ does not change sign if and only if $\tau_*>-\mu_1$. %In this case equation (???) has exactly one root larger then $-\lambda_1$ and %$\gamma(y;\tau_*)>0$ for $y\in[0,d)$.

\end{theorem}
\begin{proof} (i) this follows from the monotonicity of the function $\sigma$ proved in Proposition \ref{CAu12a}.

(ii) This assertion also follows from Proposition \ref{CAu12a}.

\end{proof}

 {\bf The choice of the frequency $\tau_*$.} If $\mu_1>0$ we  assume that $\sigma(0)<0$ and denote by $\tau_*>0$ the unique root of the equation (\ref{Okt6bb}). If $\mu_1\leq 0$ then  $\tau_*>-\mu_1$ is the unique  root of the equation (\ref{Okt6bb}).  By $\lambda_*$ we denote the corresponding value of $\lambda$.

\subsection{Transevsality condition. Small amplitude water waves}\label{SAu14ba}

As a result of bifurcation from the trivial (horizontal) free surface, the small amplitude bifurcation surface and corresponding water domain is described by
\begin{equation}\label{Au21c}
{\mathcal D}={\mathcal D}_\eta=\{(x,y)\,:\,x\in\Bbb R,\;0<y<h+\eta(x)\},\;\;{\mathcal S}={\mathcal S}_\eta=\{(x,y)\;:\,x\in\Bbb R,\;y=h+\eta(x)\},
\end{equation}
where $\eta$ is even, $\Lambda$-periodic function, which is small together with the first derivative. Furthermore we will assume that the function $\omega$ is analytic.

We use the spaces
$$
C^{k,\alpha}_{e,\Lambda}(\Bbb R)=\{\xi\in C^{k,\alpha}(\Bbb R)\,:\,\mbox{$\xi$ is even and $\Lambda$-periodic}\},
$$
$$
C^{k,\alpha}_{e,\Lambda}({\mathcal D})=\{w\in C^{k,\alpha}({\mathcal D})\,:\,\mbox{$w$ is even and $\Lambda$-periodic}\}
$$
and
$$
\widehat{C}^{k,\alpha}_{e,\Lambda}({\mathcal D})=\{w\in C^{k,\alpha}_{e,\Lambda}({\mathcal D})\,:\,\mbox{$w$ vanishing on the bottom of
${\mathcal D}$}\},
$$
where $k=0,1,\ldots$, $\alpha\in (0,1)$. We will use also another strip-like domains ${\mathcal R}$ and the space $C^{k,\alpha}_{e,\Lambda}({\mathcal R})$ and $\widehat{C}^{k,\alpha}_{e,\Lambda}({\mathcal R})$  with similar definitions as above.

\bigskip
\noindent
{\bf The case of fixed period.}
 The  constructions in Sect. \ref{SAu13a} can be done for other values of $\lambda$ also. We will use the notation $\Psi^\lambda$,  $m(\lambda)$, $Q(\lambda)$ and $\sigma(\tau;\lambda)$  to indicate the dependence of all this quantities  on $\lambda$, which is considered as the bifurcation parameter. The transversality condition required for existence of small amplitude Stokes waves is equivalent to
\begin{equation}\label{Au8a}
\sigma_\lambda (\tau_*;\lambda)\neq 0\;\;\mbox{for $\lambda=\lambda_*$},
\end{equation}
see \cite{KN14} and \cite{Var23}.
Let
\begin{equation}\label{Au18a}
X=x,\;\;\;\;Y=\frac{hy}{\eta+h}
\end{equation}
and let
\begin{equation}\label{Au21ca}
{\mathcal R}=\{(X,Y)\,:\,X\in\Bbb R,\;\;0<Y<h\}.
\end{equation}
According to \cite{Var23} (see also Theorem 1.3, \cite{KN14}, for the case non-analytic $\omega$), there exists a continuous branch
\begin{equation}\label{Au21a}
(\eta(X;t),\phi(X,Y;t),\lambda(t))\in C^{2,\alpha}_{e,\Lambda_*}(\Bbb R)\times \widehat{C}^{2,\alpha}_{e,\Lambda_*}({\mathcal R})\times\Bbb R,
\end{equation}
of solutions to (\ref{KK2a}), which can be reparametrized analytically in a neighborhood of any point on
the curve for small $t$. In $(x,y)$ variables $\psi(x,y;t)=\phi (x,\frac{hy}{\eta+h};t)$.
All  small amplitude water waves are exhausted by
$$
\psi(x,y;t)= \Psi^\lambda(y),\;\;\eta=0,\;\;Q(\lambda)=\lambda^2,m(\lambda)=-\Psi^\lambda(0),
$$
and by (\ref{Au21a}). The first terms in (\ref{Au21a}) have the following form in $(x,y)$ coordinates:
$$
\eta(x)=t\cos(\tau_*x)+0(t^2),\;\;\; \lambda(t)=\lambda_*+O(t)
$$
and
\begin{equation}\label{Au21aa}
\psi(x,y;t)=\Psi^\lambda(y)+c_*t\gamma(y;\tau_*)\cos(\tau_*x)+O(t^2),
\end{equation}
where
\begin{equation}\label{Au19b}
c_*=-\Psi^{\lambda_*}_Y(0)
\end{equation}
and $\Psi^\lambda$ and $\gamma$ are smoothly extended outside $[0,h]$ in order to define then for $y<h$.

Differentiating (\ref{Au21aa}) with respect to $x$, we get
\begin{equation}\label{Au8aa}
\psi_x(x,y;t)=\tau_*c_*t\gamma(y;\tau_*)\cos(\tau_*x)+O(t^2).
\end{equation}
This implies the following

\begin{proposition}\label{PAua} Let $\tau_*$ is chosen according to {\rm Theorem \ref{TAu8}} and let {\rm (\ref{Au8a})} be valid. Then all small solution of {\rm (\ref{K2a})}, such that $\psi_x$ is not identically zero, satisfies {\rm (\ref{Au8aa})}.

\end{proposition}

\bigskip
{\bf Variable period.} Here we assume that the Bernoulli constant is fixed and the period of the wave is chosen as a bifurcation parameter. We take
$$
\Psi(y)=\Psi^{\lambda_*}(y),\;\;Q=\lambda_*^2/2\;\; \mbox{and}\;\; m=-\Psi^{\lambda_*}(-h)
$$
 in (\ref{K2a}). %and denote by $\Lambda_0$
%Choosing the uniform stream solution we assume that $m=1$ in (\ref{XX1}) and $Q$ is defined by the relation (\ref{J11a}).
The bifurcation frequency is defined by  "{\bf The choice of the frequency $\tau_*$}" after Theorem \ref{TAu8}, where $\Psi'(0)=\lambda_*$. So in the case $\mu_1>0$ we have additional restriction on the choice of $\lambda_*$ which must guarantee the right sign of $\sigma(0)$\footnote{This corresponds to a well-known restriction on the Froude number for existence of small amplitude water waves for unidirectional flows, see \cite{Koz1}}.
Since the solution $\Psi$ does not depend on $x$ it has an arbitrary period $\Lambda>0$. Using that the frequency
$$
\tau=2\pi/\Lambda
$$
and that the dispersion equation $\sigma (\tau)=0$ has unique solution $\tau=\tau_*$, we conclude that the only bifurcation point is $\Lambda=\Lambda_*=2\pi/\tau_*$.

As it is shown in \cite{KN14} the transversality condition (\ref{Au8a}) is always satisfied.

We shall use the following change of variables
\begin{equation}\label{Au13aa}
X=\frac{\Lambda_*x}{\Lambda},\;\;;Y=\frac{h(y-\eta)}{\eta+h}.
\end{equation}

According to Theorem 1.2, \cite{KN14}, there exists a continuous branch
\begin{equation}\label{Au21b}
(\eta(X;t),\phi(X,Y;t),\Lambda(t))\in C^{2,\alpha}_{e,\Lambda_*}(\Bbb R)\times \widehat{C}^{2,\alpha}_{e,\Lambda_*}({\mathcal R})\times\Bbb R,
\end{equation}
such that this curve can be reparametrized analytically in a neighborhood of any point on
the curve for small $t$. The analytic vorticity is not considered in \cite{KN14} but now the extension to the analytic case is quite standard. Here
$$
\psi(x,y;t)=\phi(\frac{\Lambda_*x}{\Lambda},\frac{h(y-\eta)}{\eta+h};t)
$$

All small water waves are exhausted by
$$
\psi= \Psi^{\lambda_*}(Y),\;\;\eta=0,
$$
where the period $\Lambda$ is included implicitly,
and by (\ref{Au21b}).

The first terms in the asymptotics is given by
$$
\eta(x)=t\cos(\tau_*X)+0(t^2), \;\;\Lambda(t)=\Lambda_*+O(t^2)
$$
together with
\begin{equation}\label{Au21b}
\psi(x,y;t)=\Psi^{\lambda_*}(Y)+c_*t\gamma(Y;\tau_*)\cos(\tau_*X)+O(t^2),\;\
\end{equation}
where $c_*$ is given by (\ref{Au19b}).

Differentiating (\ref{Au21b}) with respect to $x$, we get
\begin{equation}\label{Au12aa}
\psi_x(x,y;t)=\tau_*c_*t\gamma(y;\tau_*)\cos(\tau_*X)+O(t^2).
\end{equation}

This implies the following
\begin{proposition}\label{PAua} Let $\tau_*$ is chosen according to {\rm Theorem \ref{TAu8}}. Then all small solution of {\rm (\ref{K2a})}, such that $\psi_x$ is not identically zero, satisfies {\rm (\ref{Au8aa})}.
\end{proposition}

\section{Applications}

In all forthcoming examples the vorticity function $\omega$ is assumed as before to be analytic.

We assume that the uniform stream solution $\Psi^{\lambda}$ is fixed and the frequency $\tau_*$ is chosen according to the rule
 {\bf The choice of the frequency $\tau_*$.} at the end of Sect. \ref{SAu13a}. We denote $\lambda_*=2\pi/\tau_*$.

\subsection{Waves with multiple critical layers}\label{SAu14b}

In paper \cite{Var23} the problem (\ref{K2a}) was considered  in the case when the free surface is given by $y=\eta(x)$, i.e. $\eta$ is $\Lambda$-periodic and ${\mathcal D}$ and ${\mathcal S}$ are defined by (\ref{Au21c}).
$$
{\mathcal D}=\{(x,y)\,:\,x\in\Bbb R,\;0<y<h+\eta(x)\},\;\;{\mathcal S}=\{(x,y)\;:\,x\in\Bbb R,\;y=h+\eta(x)\}.
$$
We assume that  the transversality condition (\ref{Au8a}) is satisfied. Then a global analytic  branch of solutions was constructed in \cite{Var23}.

We make the change of variables (\ref{Au18a}) and let ${\mathcal R}$ is given by (\ref{Au21ca}).
%$$
%\partial_x=\partial_X-\frac{h(h+y)\eta'}{(\eta+h)^2}\partial_Y,\;\;\partial_y=\frac{h}{h+\eta}\partial_Y.
%$$
The following result on  existence of the analytic branch of water waves is proved in \cite{Var23}.

The local curve obtained in Sect. \ref{SAu14ba} in "{\bf The case of fixed period}" can be
uniquely extended (up to reparametrization) to a continuous curve defined for $t\in\Bbb R$
\begin{equation}\label{Au19ba}
(\eta(X;t), \phi(X,Y;t),\lambda(t))\;\in\;C^{2,\alpha}_{e,\Lambda_*}(\Bbb R)\times \widehat{C}^{2,\alpha}_{e,\Lambda_*}({\mathcal R})\times\Bbb R
\end{equation}
of solutions to (\ref{K2a}), such that the following properties hold:

(i) The curve can be reparametrized analytically in a neighborhood of any point on
the curve.

(ii) The solutions are even and have the frequency $\tau_*$
for all $t\in \Bbb R$.

(iii) One of the following alternatives occur

(A) There exists subsequences $\{t_n\}_{n\in N}$, with
$t_n\to\infty$, along which at least one of (i) the solutions are unbounded, (ii) the surface
approaches the bed, or (iii) surface stagnation is approached, hold true.

(B)  or the curve is closed.

\bigskip
We will discuss the last alternative (B). For this goal we introduce
the mapping
\begin{equation}\label{Au19a}
(x,y)\to (X(x),Y(x,y)).
\end{equation}
One can verify that it is an diffeomorphism from $\overline{D}$ to $\overline{R}$, where
$$
R=(0,\Lambda/2)\times (0,h)
$$
 This isomorphism satisfies all conditions from Sect. \ref{SAu14c}.

\begin{theorem}\label{TAu19a} If $\Psi_y(0)>0$ ($\Psi_y(0)<0$) then
the function $\psi_x=\psi_x(x,y;t)$ ($-\psi_x$) satisfies properties {\rm (i)-(iii) (\ref{Au16a})} for any interval containing small positive $t$ and which has no points $t$ where $\psi_x(x,y;t)=0$ identically.
\end{theorem}
\begin{proof} We consider the case $\Psi_y(0)>0$. the case $\Psi_y(0)<0$ is considered similarly. Denote by $V=V(X,Y;t)$ the image of the function $\psi_x$ under the mapping (\ref{Au19a}). By Sect. \ref{SAu14ba} the function $\psi_x$ satisfies all conditions (i)-(iii), (\ref{Au16a}), for small $t$. By Lemma \ref{L2} the same is true for the function $V$ in the domain $R$.

 Let the assertion of theorem is proved for a certain $t_*> 0$. Then by Proposition \ref{PAu17a} the assertion is true for $v$ close to $V(X,Y;t_*)$ in $C^2(\overline{R})$. In particular there exist an interval $(t_*-\epsilon,t_*+\epsilon)$ such that $V(X,Y;t)$ also satisfies (i)-(iii) in (\ref{Au16a}). By Lemma \ref{L2} the same is true for $\psi(x,y;t)$ for $|t-t_*|<\epsilon$.

% Let us show that it is also true for $t\in (t_*-\varepsilon,t_*+\varepsilon)$ where $varepsilon$ is a small positive number. By Lemma ??? this %implies that the function $\Phi(X,Y;t_*)$ satisfies (i)-(iii) for the domain ${\mathcal R}$.
%Since $\phi_X(X,Y;t)$ is zero for $X=0$, $X=\Lambda/2$, $Y=0$ and the normal derivatives there are non-vanishing we get that the same is true for %the functions $\phi_X$ for $t$ close to $t_*$. This implies the positivity of $\phi$ in a neighborhood of $\partial $ for $t$ close to $t_*$. The %positivity of $\phi$ inside ${\mathcal R}$ outside a small neighborhood of the boundary $\partial {\mathcal R}$ is quite evident.

 Assume that the assertion is true for the function $\psi_x$ in a certain interval $(0,t_*)$. Then the function $V$ satisfies (i)-(iii) on the same interval by Lemma \ref{L2}. Therefore the function $V\geq 0$ for  $t=t_*$. This implies that the same is true for the function $\psi_x$. Applying Lemma \ref{L1}, we conclude that the assertion is valid for the function $\psi_x$ at  $t_*$.

\end{proof}

\begin{corollary}\label{Cor19} The following assertions hold.

{\rm (i)} If (\ref{Au19ba}) is a closed curve then there are at least two points $\lambda_*$ and $\lambda_1\neq \lambda_*$ such that $(\lambda_*,\Psi^{\lambda_*}(y),0)$ and $(\lambda_1,\Psi^{\lambda_1}(y),0)$ belong to the curve.

{\rm (ii)} If the only uniform stream solution on the curve (\ref{Au19ba}) is $(\lambda_*,\Psi^{\lambda_*}(y),0)$
 then the curve {\rm (\ref{Au19ba})} is not closed and

 If $\Psi^{\lambda_*}(0)>0$ then $\eta$ is decreasing on $(0,\Lambda/2)$ for $t>0$;

 If $\Psi^{\lambda_*}(0)<0$ then $\eta$ is increasing on $(0,\Lambda/2)$ for $t>0$.
\end{corollary}
\begin{proof} If there is only one uniform stream solution on the bifurcation curve then $\psi_x$ has the same sign on the curve outside this point but due to analyticity it has different sign for $t>0$ and $t>0$.

The second assertion follows from Theorem \ref{TAu19a}.

\end{proof}

\subsection{Overhanging free surface}

Here we consider a more general class of domains.
We assume that the free surface boundary ${\mathcal S}$ is given in a parametric form : $(x,y)=(u(s),v(s))$, $s\in\Bbb R$, where both functions are $\Lambda$-periodic, $v(s)>0$ and $u(s)\to\pm\infty$ when $s\to\pm\infty$. We assume also that the curve does not intersect itself. Here we use the same change of variables as in \cite{Wa} and \cite{CSrVar}.
%The following problem is considered in \cite{Wa}
%\begin{eqnarray}\label{Au6a}
%&&\Delta \psi+\omega(\psi)=0\;\;\mbox{in ${\mathcal D}$},\nonumber\\
%&&\frac{1}{2}|\nabla\psi|^2+\eta=Q\;\;\mbox{on ${\mathcal S}$},\nonumber\\
%&&\psi=0;\;\mbox{on ${\mathcal S}$},\nonumber\\
%&&\psi=-m\;\;\mbox{for $y=-h$}.
%\end{eqnarray}

Let
$$
{\mathcal R}_{-h} :=\{(X,Y)\,:\,X\in {\Bbb R},\;-h<Y< 0\}
$$
be a strip with depth $h$.

Introduce $C_h^{\Lambda_*}$ as the $\Lambda_*$-periodic Hilbert transform
given by
$$
C_h^{\Lambda_*} u(X)=-i\sum_{k\neq 0}\coth (k\tau_* h)\widehat{u}_k e^{ik\tau_* X},
$$
for an $\Lambda_*$-periodic function
$$
u(X) = \sum_{k\neq 0}\widehat{u}_ke^{ik\tau_* X}
$$
with zero average, appears. The following assertion on a conformal change of variables can be found in \cite{Wa}

\begin{lemma}
{\rm (i)} There exists a unique positive number $h$ such that there exists a conformal mapping
$H = U + iV$ from the strip ${\mathcal R}$
to ${\mathcal D}$ which admits an extension as a
homeomorphism between the closures of these domains, with $\Bbb R\times \{0\}$ being mapped
onto ${\mathcal S}$ and ${\Bbb R} \times \{-h\}$ being mapped onto $\Bbb R \times \{0\}$, and such that $U(X + \Lambda_*, Y) =
U(X, Y) + \Lambda_*$, $V (X + \Lambda_*, Y) = V (X, Y)$, $(X, Y)\in {\mathcal R}$.

{\rm (ii)} The conformal mapping $H$ is unique up to translations in the variable $X$ (in the
preimage and the image)

{\rm (iii)} $U$ and $V$ are (up to translations in the variable $X$) uniquely determined by $w=V(\cdot,0)-h$ as follows: $V$ is the unique ($\Lambda_*$-periodic) solution of
\begin{eqnarray*}
&&\Delta V = 0\;\;\mbox{ in ${\mathcal R}$}, \\
&&V = w + h\;\; \mbox{on $Y = 0$},\\
&&V = 0\;\; \mbox{on Y =-h},
\end{eqnarray*}
and $U$ is the (up to a real constant unique) harmonic conjugate of $-V$. Furthermore,
after a suitable horizontal translation, ${\mathcal S}$ can be parametrised by
\begin{equation}\label{Au13a}
{\mathcal S} = \{(X + (C^L_hw)(X), w(X) + h)\, :\, X\in {\Bbb R}\}
\end{equation}
and it holds that
$$
S\nabla V = (1 + C^L_hw', w')
$$

{\rm (iv)} If ${\mathcal S}$ is of class $C^{1,\beta}$ for some $\beta > 0$, then $U, V \in C^{1,\beta}({\mathcal R})$ and
$$
|dH/dz|^2 = |\nabla V |^2\neq 0\;\; \mbox{in}\; \overline{\mathcal R}
$$
\end{lemma}

Thus the vector function $H$ delivers  diffeomorphism
$$
x=U(X,Y),\;\;\;y=V(X,Y)\;\;\mbox{mapping ${\mathcal R}_{-h}\;\to \;{\mathcal D}$}.
$$

We use the same bifurcation parameter $\lambda\in\Bbb R$ as in Sect. \ref{SAu14b}. The trivial solution corresponding to $\lambda$ is then $\Psi=\Psi^\lambda(y)$.

%The constant $Q$ is expressed through $\lambda$ and $\phi$, $w$.

In the paper \cite{Wa} it is proved the existence of a branch
$$
\phi(X,Y;s),\;w(X;s),\;\lambda(s)\in C^{2,\alpha}_{e,\Lambda_*}\times C^{1,\alpha}_{e,\Lambda_*}\times\Bbb R,
$$
which has a real-analytic reparametrisation locally around each of its points, such that the function
$$
\psi(x,y:s)=\phi(X(x,y;s),Y(x,y;s);s)\;\;\mbox{solves (\ref{K2a}) in ${\mathcal D}_s$}.
$$
Here ${\mathcal D}_s$ is the domain between $y=0$ and ${\mathcal S}_s$, where ${\mathcal S}_s$ is given by (\ref{Au13a}).

Moreover this analytic branch is the extension of the branch of small amplitude water waves constructed in Sect. \ref{SAu14b}.

It is shown in \cite{Wa} (see also \cite{CSrVar} for the case of the constant vorticity) that one of the
 following alternatives occur:

(i) this alternative describes typical behavior of the branch for large values of $s$, see Introduction (a)-(c);

(ii) the second alternative says that there exists a closed curve.% i.e. there exists $T > 0$ such that $(\lambda^{s+T}, w^{s+T},\phi^{s+T}) = %(\lambda^s, w^s, \phi^s)$ for all $s\in\Bbb R$.

\bigskip
Let us discuss the second alternative.

The same arguments can show that
Theorem \ref{TAu19a}  is true in this more general case. This means that the alternative (ii) can be replace

(ii)' There are at least two points $\lambda_*$ and $\lambda_1\neq \lambda_*$ such that $(\Psi^{\lambda_*}(y),0,\lambda_*)$ and $(\Psi^{\lambda_1}(y),0,\lambda_1)$ belong to the curve (\ref{Au19ba}).

Corollary \ref{Cor19} can be modified to the following

\begin{corollary}\label{Cor19z} The following assertions hold.

{\rm (i)} If {\rm (\ref{Au19ba})} is a closed curve then there are at least two points $\lambda_*$ and $\lambda_1\neq \lambda_*$ such that $(\lambda_*,\Psi^{\lambda_*}(y),0)$ and $(\lambda_1,\Psi^{\lambda_1}(y),0)$ belong to the curve.

{\rm (ii)} If the only uniform stream solution on the curve (\ref{Au19ba}) is $(\lambda_*,\Psi^{\lambda_*}(y),0)$
 then the curve {\rm (\ref{Au19ba})} is not closed and

 If $\Psi^{\lambda_*}(0)>0$ ($\Psi^{\lambda_*}(0)<0$) then $\psi_x>0$ ($\psi_x<0$) on $S$  for $t>0$.

\end{corollary}

Since the inclination of the bondary curve is defined by $-\psi_x/\psi_y$,  the positivity property in Corollary \ref{Cor19z}(ii) can be used to exclude certain configurations of the limit self-intersections of the boundary curves (see \cite{CSrVar}).

\subsection{Water waves with variable period}

Let  the variables $X$ and $Y$ are given by (\ref{Au13aa}).
%If $(\Psi, h)$ is defined by $Q$ and assumption (II) holds for Equation (11) corresponding
%to this stream solution, then there exist $\varepsilon > 0$ and a continuously differentiable
%mapping from $ t\in (-\varepsilon,\varepsilon)$ to a neighbourhood of $(0, 0)$ in $\Bbb R\times\Pi^{2,\alpha}_0$.

The local curve obtained in Sect. \ref{SAu14ba} can be
uniquely extended (up to reparametrization) to a continuous curve defined for $t\in\Bbb R$
\begin{equation}\label{Au10c}
(\phi(X,Y;t),\eta(X;t),\Lambda(t))\;\in\;C^{2,\alpha}({\mathcal R})\times C^{2,\alpha}(\Bbb R)\times\Bbb R
\end{equation}
of solutions to (\ref{K2a}), such that the following properties hold:

(i) The curve can be reparametrized analytically in a neighborhood of any point on
the curve.

(ii) The solutions are even and have wavenumber $\tau_*$,
for all $t\in \Bbb R$.

The following asymptotics for small $f$ hold
$$
\eta(X;t)=t \cos(\tau_* X)+O(t^2),\;\;\Lambda(t)=\Lambda_*+O(t^2),\;\;\phi(X,Y;t)=\Psi^{\lambda_*}(Y)+tc_*\gamma(Y;\tau_*)\cos(\tau_*X)+O(t^2),
$$
where $c_*$ is the same as in (\ref{Au19b}).

(iii) One of the following alternatives occur

(A) There exists a subsequence $\{t_n\}_{n\in N}$, with
$t_n\to\infty$, along which at least one of (i) the solutions are unbounded, (ii) the surface
approaches the bed, (iii) the period tends to $0$ or to $\infty$, or (iv) surface stagnation is approached, hold true.

(B)  or the curve is closed.

\bigskip
The above assertion can be proved in same way as similar assertion in \cite{Var23}, see Sect. \ref{SAu14b} for its formulation. The only difference here is that the bifurcation parameter is the period and in verification of compactness and Fredholm properties in the global bifurcation theorem must be changed a little bit. We refer also to \cite{KL1} and \cite{KL3} where the assertion is proved for the unidirectional flows.

\bigskip

Let us discuss the option (B).

First we note that Theorem \ref{TAu19a} is valid with the same proof. Corollary \ref{Cor19} can be improved in the part (i) as follows

\begin{corollary} The following assertions hold.

{\rm (i)} The curve is not closed. The only uniform stream solution on this curve is met when $t=0$. So the alternative (B) can be excluded.

{\rm (ii)} If $\Psi^{\lambda_*}(0)>0$ ($\Psi^{\lambda_*}(0)<0$) then $\eta$ is decreasing (increasing) on $(0,\Lambda/2)$ for $t>0$

\end{corollary}
\begin{proof} Since the option (i) in Corollary \ref{Cor19} is excluded, see Sect. \ref{SAu14ba} (variable period), we see that  Corollary \ref{Cor19}(i) can be replaced by assetion (i) from this corollary.

The proof of (ii) is the same as in  Corollary \ref{Cor19}.
\end{proof}

\bigskip
\noindent {\bf Acknowledgements.} V.~K. was supported by the Swedish Research
Council (VR), 2017-03837.

\section{References}

{

\end{document}